\documentclass[10pt,oneside]{amsart}
\usepackage{amssymb, amscd, amsmath, amsthm, latexsym, hyperref,ams refs}
\usepackage{pb-diagram, fancyhdr, graphicx, psfrag, enumerate}
\usepackage[text={6.3in,8.5in},centering,letterpaper,dvips]{geometry}
\usepackage{pinlabel}
\usepackage{enumitem}\setlist[enumerate,1]{font=\upshape, itemsep=1ex}\setlist[enumerate,2]{font=\upshape}

\def\Z{{\mathbb Z}}

\def\Q{{\mathbb Q}}

\def\calc{\mathcal{C}}
\def\calm{\mathcal{M}}

\def\calt{\mathcal{T}}
\def\caln{\mathcal{N}}

\def\cs{\mathbin{\#}}
\DeclareMathOperator{\cfk}{\rm CFK}
\newcommand{\spinc}{\ifmmode{{\mathfrak s}}\else{${\mathfrak s}$\ }\fi}
\newcommand{\spinct}{\ifmmode{{\mathfrak t}}\else{${\mathfrak t}$\ }\fi}

\newcommand{\fig}[2] { \includegraphics[scale=#1]{#2} }

\def\la{\widetilde{\alpha}}
\def\lb{\widetilde{\beta}}

\newtheorem{theorem}{Theorem}[section]
\newtheorem{theorem-app}{Theorem}[section]
\newtheorem{lemma}[theorem]{Lemma}
\newtheorem{corollary}[theorem]{Corollary}
\newtheorem{conjecture}{Conjecture}

\theoremstyle{definition}

\begin{document}
\title[Knot reversal and concordance]{Knot reversal acts non-trivially on the concordance group of topologically slice knots}

\author{Taehee Kim}
\address{Taehee Kim: Department of Mathematics\\
	Konkuk University\\
	Seoul 05029\\
	Republic of Korea
}
\email{tkim@konkuk.ac.kr}

\author{Charles Livingston}
\address{Charles Livingston: Department of Mathematics, Indiana University, Bloomington, IN 47405}
\email{livingst@indiana.edu}

\thanks{The first  author was supported by Basic Science Research Program through the National Research Foundation of Korea (NRF) funded by the Ministry of Education (no.2018R1D1A1B07048361). The second   author was supported by a grant from the National Science Foundation, NSF-DMS-1505586.}


\begin{abstract}  
We construct an infinite family of  topologically slice knots that are not smoothly concordant to their reverses.  More precisely, if $\calt$ denotes the concordance group of topologically slice knots and $\rho$ is the involution of $\calt$ induced by string reversal, then $\calt / \text{Fix}(\rho)$ contains an  infinitely generated free subgroup.     The result remains true modulo the subgroup of $\calt$ generated by knots with trivial Alexander polynomial. 
\end{abstract}

\maketitle

\section{Introduction}

For an oriented knot $K$ in $S^3$, denote by $\rho(K)$ the knot formed from $K$ by reversing its   string orientation.   Note that $\rho(K)$ is not necessarily the inverse of $K$ in the concordance group, so we call it the {\it reverse} of $K$  rather than use the earlier terminology, the {\it inverse} of $K$. Clearly, $\rho$ is an involution on the set of knots; a proof that $\rho$ is nontrivial eluded knot theorists until Trotter~\cite{MR0158395} published {\it Non-invertible knots exist} in 1963.  Further advances were presented in such work as~\cite{MR1395778, MR543095, MR559040}.  With the advent of computer programs such as SnapPy~\cite{SnapPy},   determining if a given knot is reversible  is now routine. 

It is evident that  $\rho$ induces an involution on the smooth knot concordance group $\calc$; to avoid burdensome notation, we will use the same symbol,  $\rho$, to denote this induced  involution.  The nontriviality of this action was proved in~\cite{ MR1670424};  see also the earlier reference~\cite{MR711524} which has a small gap, the resolution of which is contained in ~\cite{MR1670424}.

If one restricts to the concordance group of topologically slice knots, $\calt \subset \calc$, the situation becomes  more difficult.  Casson-Gordon invariants have provided the only tools  used to study the interplay between concordance and reversibility, and these vanish for knots in $\calt$.   Heegaard Floer invariants, which in general  offer powerful tools for working with knots in $\calt$, appear to be insensitive to string orientation.  For instance, the Heegaard Floer knot chain complexes $\cfk^\infty(K)$ and $\cfk^\infty(\rho(K))$ are filtered chain homotopy equivalent.  In addition, concordance invariants arising from Khovanov homology such as the Rasmussen invariant~\cite{MR2729272} do not detect string orientation.  Despite these challenges, we prove that the action of $\rho$ on $\calt$ is highly nontrivial: let $\text{Fix}(\rho)$ denote the fixed set of the involution.

\begin{theorem}\label{thm:main} The quotient $\calt / \text{Fix}(\rho)$ contains an infinitely generated free subgroup.
\end{theorem}

Let $\calt_\Delta$ denote the subgroup of $\calt$ consisting of concordance classes represented by knots with trivial Alexander polynomial. It was first proved in \cite[Theorem~A]{MR2955197} that $\calt/\calt_\Delta$ is nontrivial and furthermore contains an infinitely generated free subgroup. Theorem~\ref{thm:main} is an immediate corollary of the following stronger theorem, which also extends \cite[Theorem~A]{MR2955197} .

\begin{theorem}\label{thm:main2}
The quotient $\calt / (\text{Fix}(\rho)+\calt_\Delta)$ contains an infinitely generated free subgroup.
\end{theorem}

Each of the knots $K$ constructed in the proof of Theorem~\ref{thm:main2} has the property that  $K \cs -\rho(K)$ is not smoothly slice, whereas $K \cs - K$ is.  As observed by Kearton~\cite{MR929430}, these knots are Conway mutants.  In general, knot invariants tend not to distinguish a knot from its Conway mutant~\cite{MR0258014};  a very short sampling of related references include~\cite{MR2195064, MR906585, MR2657370, MR2592723, MR2657645}.  References for the application of Heegaard Floer methods to mutation (but not in the setting of concordance or string reversal) include~\cite{MR2058681, MR3325731,   MR3874002, 2017arXiv170100880L}.    Recent work that touches upon Conway mutation and concordance includes~\cite{MR3402337, MR3660096}, and especially the breakthrough result of Piccirillo~\cite{Piccirillo:2018aa} proving that the Conway knot is not slice.

\smallskip

\noindent{\it Outline.}  In Section~\ref{sec:obstructions} we give slicing obstructions obtained by combining Casson-Gordon invariants and the Heegaard Floer $d$--invariant. In Section~\ref{sec:single knot}  we present a specific topologically slice knot $K$ and prove that $K \cs -\rho(K)$ is not smoothly slice. This knot $K$ is similar to one used in~\cite{MR3109864}; there, the linking form of the 3--fold branched cover of $S^3$ branched over $K$ has exactly two metabolizers.  Separate arguments are applied related to each metabolizer, one using Casson-Gordon theory and the other  Heegaard Floer theory.  In the current setting, the relevant branched covering has a much larger number of  metabolizers (76 to be precise) and many of these do not offer obstructions to sliceness.  Thus, we  first  eliminate many from consideration, leaving four distinct families to consider.   Once that  is done, topological obstructions are derived from  invariants developed in~\cite{MR1162937}; we build our computations of the relevant Heegaard Floer invariants using a specific computation of~\cite{MR3109864}, but  more detail is required because that paper did not address an issue of Alexander polynomial one knots which we want to include here.

In building this single example in Sections~\ref{sec:obstructions} and \ref{sec:single knot}, we are able to develop the key tools and notation for the general problem. Then, in Section~\ref{sec:infinite family} we build an infinite family of knots used in proving Theorem~\ref{thm:main2}. A key ingredient is to find infinitely many topologically slice knots $K_i$ such that $K_i$ are nontrivial in $\calt / (\text{Fix}(\rho)+\calt_\Delta)$ and the orders of the first homology groups of the 3--fold branched covers of $S^3$ branched over $K_i$ are relatively prime, which is done using certain number theoretic arguments (see Appendix~\ref{sec:appendix}). Another key ingredient is computations of the Heegaard Floer $d$--invariants of the $K_i$, and 
this is accomplished using the powerful methods developed by Cha~\cite{arxiv:1910.14629}.  

\smallskip

\noindent{\it Acknowledgements.}   Conversations with Jae Choon Cha motivated us to reexamine the problem of reversibility in concordance.  Although his work with Min Hoon Kim~\cite{Cha:2017aa} is not used explicitly, it was through that work that we were led to our successful approach.  Conversations with Pat Gilmer, Se-Goo Kim and Aru Ray were also of great value.

\section{Slicing Obstructions}\label{sec:obstructions}
\subsection{Casson-Gordon invariants} Let $Y_q(K)$ denote the $q$--fold cyclic branched cover of $S^3$ with branch set an arbitrary  knot $K$; we will henceforth assume that  $q$ is an odd prime power.  It is then the case that $Y_q(K)$ is a $\Q$--homology sphere.

For each element $\chi \in H_1(Y_q(K))$ there is a Casson-Gordon invariant $\eta(K, q, \chi )$.   This  invariant  takes values in a Witt group. Later we will describe computable  invariants of this Witt group that provide slicing obstructions, and thus we will not need the precise definition of the group itself.  The invariant $\eta$ was defined in~\cite{MR900252}, where it was denoted~$\tau$.  In that original work, $\chi$  was an element of $\text{Hom}(H_1(Y_q(K)), \Z_{p^r})$ for some prime power $p^r$.  We have chosen $\chi \in H_1(Y_q(K))$; via the nonsingular linking form on $H_1(Y_q(K))$, such a $\chi$ determines a homomorphism in $\text{Hom}(H_1(Y_q(K)), \Q/\Z)$.  By restricting to elements of prime order $p$, the image of the homomorphism is in $\Z_p$, as desired.   We will use Gilmer's theorem~\cite{MR711523} that $\eta$ is additive:  $\eta(K\cs K', q, \chi \oplus \chi' )= \eta(K\, q, \chi ) +\eta( K', q,  \chi' )$.  

\subsection{Heegaard Floer invariants}   If $Y_q(K)$ is a $\Z_2$--homology sphere, there is    a Heegaard Floer invariant $\bar{d}(Y_q(K), \chi )$.   Here we will summarize our notation and some of the essential properties of this invariant; further details will appear later in the exposition.
The Heegaard Floer $d$--invariant, defined in~\cite{MR1957829},  takes values in $\Q$.   It is usually expressed as $d(Y, \spinc)$, where $Y$ is a 3--manifold and  $\spinc$ is a Spin$^c$--structure.  In the setting of $\Z_2$--homology spheres,  Spin$^c$--structures correspond to elements of $H^2(Y) \cong H_1(Y)$, so we will work with the first homology rather than with Spin$^c$.  We then have the definition $\bar{d}(Y,\chi) =d(Y,\chi) -d(Y,0)$. The use of $\bar{d}$ to address issues related to the presence of knots with trivial Alexander polynomial first appeared in~\cite{MR2955197}.
We will use the additivity property $\bar{d}(Y\cs Y',\chi \oplus \chi')=\bar{d}(Y,\chi)+ \bar{d}( Y', \chi')$.  Note that $\bar{d}(Y, 0) = 0$.  One key  result states that if $H_1(Y, \Z_2) = 0$ and $Y = \partial W$, where $W$ is a rational homology four-ball and $\chi $ is   the image of a class in $H_2(W,Y)$, then $d(Y,\chi) = \bar{d}(Y, \chi) = 0$.

\subsection{Obstructions}
The main facts about the   invariants $\eta$ and $\overline{d}$ that we need are stated in the next theorem.

\begin{theorem} \label{thm:obstruct} If $K$ is smoothly slice and $H_1(Y_q(K), \Z_2) =0$, then   there is a subgroup $\calm \in H_1(Y_q(K))$ with the following four properties:  (1) $\calm$  is a metabolizer for the linking form; (2) $\calm$ is invariant under the order $q$ deck transformation of $Y_q(K)$;  
(3) For all $\chi \in \calm$, $\bar{d}(Y_q(K), \chi) =0$; (4) For all $\chi \in \calm$ of prime power order, $\eta( K, q,  \chi ) = 0$. \end{theorem}

Recall that a metabolizer for $H_1(Y_q(K))$ is a subgroup $\calm$ satisfying $\calm = \calm^\perp$ with respect to the nonsingular linking form on $H_1(Y_q(K))$.  With regards to the conditions on the Casson-Gordon theorem, this result is essentially as it appeared in~\cite{MR900252}; the equivariance of $\calm$ was noted, for instance, in~\cite{MR1670424}.  The use of $d$--invariants of covers to obstruct slicing was initiated in~\cite{MR2363303}. Notice that in Theorem~\ref{thm:obstruct} we actually have a stronger result that $d(Y_q(K), \chi) =0$ for all $\chi \in \calm$; we  use that $\bar{d}(Y_q(K), \chi) =0$, because this is the needed   slicing obstructions when working modulo $\calt_\Delta$ (see Theorem~\ref{them:modulo trivial} below).
\subsection{Working modulo $\calt_\Delta$}\label{subsec:poly1} 

Suppose $L$ is a knot with trivial Alexander polynomial. Then, we have that $H_1(Y_q(L) )= 0$.  Theorem~\ref{thm:obstruct} will be applied to provide a slicing obstruction.  Since the first homology of $Y_q(L)$ is trivial,  the presence of $L$ does not affect the values of  the $\bar{d}$--invariants or the $\eta$--invariants that we are considering.  Thus $K\cs -\rho(K)\cs L$ is not smoothly slice if we can obstruct $K \cs -\rho(K)$ from being smoothly slice using $\eta$ and $\bar{d}$. We state this as a theorem.

\begin{theorem}\label{them:modulo trivial} If $L$ is a knot with trivial Alexander polynomial and  Theorem~\ref{thm:obstruct} obstructs a knot $K \cs -\rho(K) $ from being smoothly slice, then $K \cs -\rho(K)\cs   L$ is not smoothly slice.
\end{theorem}

\section{A single example}\label{sec:single knot}
In this section  we construct  a knot $K$ that is nontrivial in the quotient group $\calt / (\text{Fix}(\rho)+\calt_\Delta)$. 

Figure~\ref{fig:knot} offers a schematic illustration of a knot $R_1$.  More generally, we let $R_n$ denote the similarly constructed knot for which there are $2n+1$ half twists between the two bands.  To simplify notation for now, we abbreviate $R_1$ by $R$ in this section.  We will specify a string orientation for $R$ later.  The construction of $K$ is fairly standard.  By appropriately replacing neighborhoods of the  curves $\alpha$ and $\beta$ with the complements of knots $J_\alpha$ and $J_\beta$, one constructs a new knot denoted $R(J_\alpha, J_\beta)$.  In effect, the bands in the evident Seifert surface for $R$ have the knots $J_\alpha$ and $J_\beta$ placed in them.   To make the notation more concise, we will sometimes abbreviate $R(J_\alpha, J_\beta)$ as $R_\ast$. 

\begin{figure}[h]
\labellist
\pinlabel $\alpha$ at 10 365
\pinlabel $\beta$ at 640 365
\endlabellist
\fig{.2 }{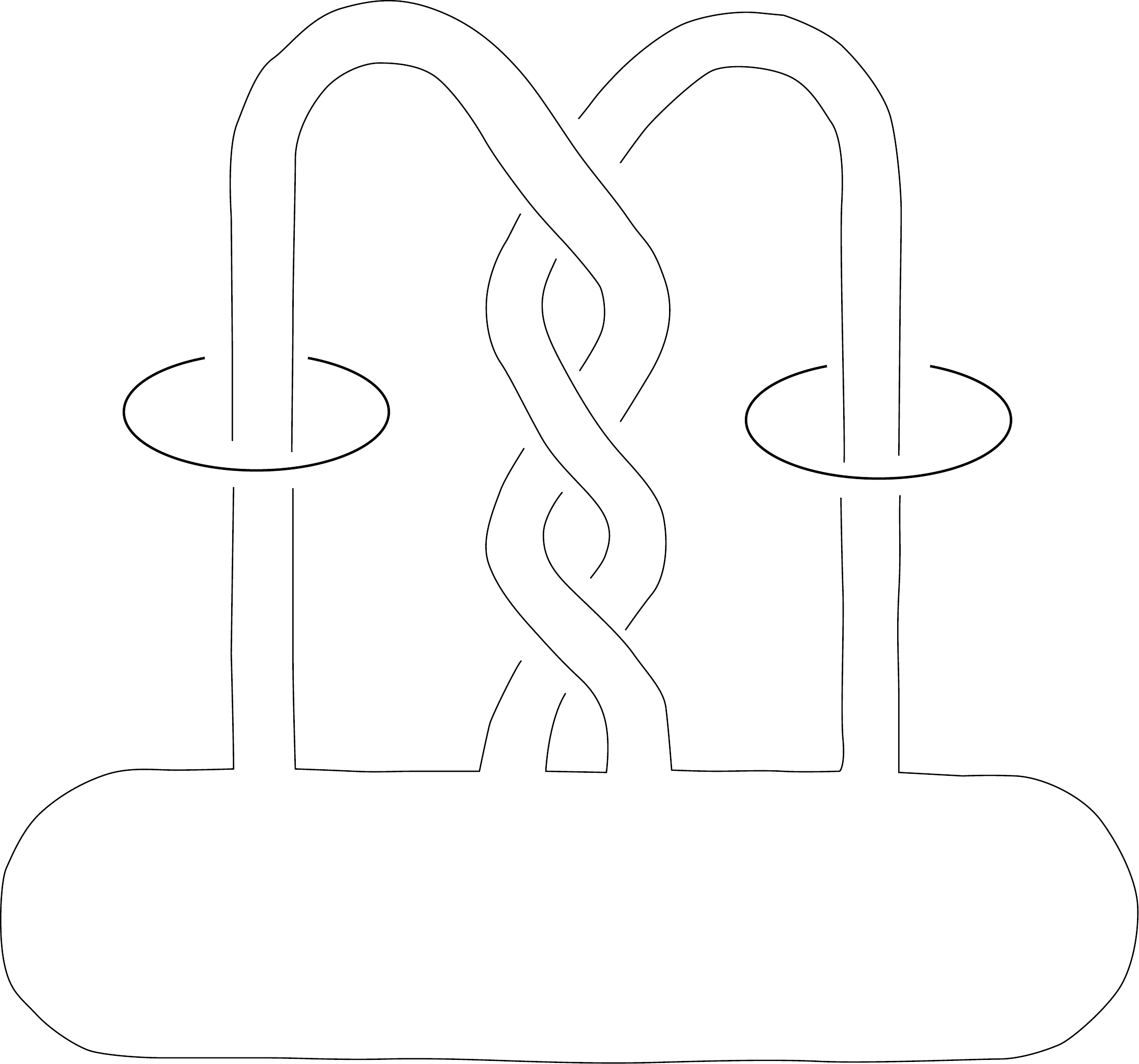}  
\caption{Knot $R_1$}
\label{fig:knot}
\end{figure}

Let $D$ be the knot $Wh(T(2,3), 0)$, the positively clasped, untwisted Whitehead double of the right-handed trefoil knot  $T(2,3)$.  Let $J$ be the knot $Wh(U,5)$, the  positively  clasped $5$--twisted Whitehead  double of unknot, having Seifert matrix 
\[   \begin{pmatrix} -1 & 1 \\
0 & 5 \\
\end{pmatrix}
\]
and Alexander polynomial $5t^2 - 11t+5$.  Our desired knot $K$ is $R(D, J)$:

\begin{theorem}\label{thm:single}  The knot $R(D, J)\ne 0 \in \calt / (\text{Fix}(\rho)+\calt_\Delta)$.  
\end{theorem}
The rest of this section is devoted to proving Theorem~\ref{thm:single}. Let $K=R(D,J)$. The knot $D$ has Alexander polynomial $\Delta_D(t) = 1$.  According to Freedman's theorem~\cite{MR679066, MR1201584},   $D$ is topologically slice.  A standard argument then shows that  $K$ is also topologically slice: $K  \in \calt$.

To prove Theorem~\ref{thm:single} it suffices to show the following: for any knot $L$ with $\Delta_L(t)=1$, 
\[
K\cs - \rho(K)\cs L \ne 0  \in \calt.
\]
By Theorem~\ref{them:modulo trivial}, we only need to show the following theorem:
\begin{theorem}\label{thm:single2} Theoreom~\ref{thm:obstruct} obstructs the knot $K\cs - \rho(K)$ from being smoothly slice. 
\end{theorem}
The following subsections  present the proof of  Theorem~\ref{thm:single2}.


\subsection{The homology of the branched cover}\label{subsec:branch}
We will now work exclusively with  $q = 3$.  Recall that we are using the  abbreviation    $R_* = R(J_\alpha,J_\beta)$.  A standard knot theoretic computation shows that for arbitrary $J_\alpha$ and $J_\beta$, $H_1(Y_3(R_*))\cong \Z_7 \oplus \Z_7$, generated by  $\la$ and $\lb$, chosen lifts of the $\alpha $ and $\beta$. 
Furthermore, viewing $H_1(Y_3(R_*))$ as a vector space over $\Z_7$, the first homology group splits into a $2$--eigenspace $E_2$ and a $4$--eigenspace  $E_4$ with respect to the order three deck transformation of $Y_3(R_*)$.   We have not yet noted the choice of orientation of $R_*$.  For one choice, which we now make, we have that $E_2$ is generated by $\la$ and $E_4$ is generated by $\lb$.

With respect to the $\Z_7$--valued  linking form,  $\la$ and $\lb$ are eigenvectors and thus $\rm{lk}(\la, \la) =  0 = \rm{lk}(\lb, \lb) $.  By replacing a generator with a multiple, we can assume $\text{lk}( \la, \lb) = 1$.

If $m$ is an oriented meridian for $R_*$, then $m$ is also an oriented meridian for $-R_*$.  (Note: $-R_*$ is built by reversing the ambient orientation of $S^3$ and then reversing the orientation of $R_*$.  The effect is to reverse the meridian twice.) In particular, $Y_3(R_*)$ and $Y_3(-R_*)$ are the same space with the same deck transformation. In particular, $H_1( Y_3(-R_*))$ has the same  splitting into  eigenspaces, $E_2 \oplus E_4$, which are generated by $\la$ and $\lb$.  Reversing the orientation of $R_*$ has the effect of inverting the deck transformation, so $H_1( Y_3(-\rho(R_*)))$  splits as a direct sum of a  2--eigenspace $E_2'$ and a 4--eigenspace $E_4'$, generated by $\lb$ and $\la$, respectively.  (That is, the roles of $\la$ and $\lb$ have been reversed.)  Henceforth, when we are working with $\rho(R_*)$, we will write $E_2'$, generated by $\lb'$, and  $E_4'$, generated by $\la'$.

We now consider the action of the deck transformation on $H_1( Y_3(R_* \cs -\rho(R_*)))$.  It has minimal polynomial $(t-2)(t-4)$.  Thus, any invariant $\Z_7$--subspace $\calm$ of $H_1( Y_3(R_* \cs -\rho(R_*)))$ splits into eigenspaces.  Here are all the possibilities.

\begin{lemma}\label{lem:cases}  The set of all equivariant metabolizers of $H_1(Y_3( R_* \cs -\rho(R_*)))$ are given by the following spans: 

\begin{enumerate}

\item  $\left< \la, \lb'\right>$; the $2$--eigenspace.

\item  $\left< \lb, \la'\right>$; the $4$--eigenspace.

\item   $\left< \la, \la'\right>$ or $\left< \lb, \lb'\right>$; one ``pure'' 2--eigenvector and one ``pure'' 4--eigenvector.

\item  $\left< \la + r \lb',  \lb + r^{-1} \la'\right>$, where $r \ne 0 \in \Z_7$.

\end{enumerate}

\end{lemma}

\begin{proof}  Cases (1) and (2) reflect the possibility that $\calm$ is a 2--dimensional eigenspace.  The alternative is that $\calm$ contains a 2--eigenvector and a 4--eigenvector.  In general, these would be spanned by vectors of the form $x\la + y\lb'$ and $z\lb + w\la'$.  The condition that these have linking number 0 is given by $xz - yw = 0 \mod 7$.    If $x\ne 0$, then by taking a multiple we can assume $x=1$.  Similarly, if $z \ne 0$, we can assume $z=1$.  With this, reducing to cases (3) and (4) is straightforward.
\end{proof}

To complete the proof of  Theorem~\ref{thm:single2}, we need to show that slicing obstructions arising from each of the metabolizers in Lemma~\ref{lem:cases} are nonzero.  The proof of this depends on additivity and the computation of specific values of invariants.  We will be  able to restrict our attention to a single summand by using the next lemma.  Notice that  string orientation is not relevant to these equations.    The following result is then seen to be trivial; it simply states that reversing the orientation of a space changes the sign of the relevant invariants.

\begin{lemma}\label{lem:meta} We have the following equalities:

\begin{enumerate}

\item $\eta(-\rho(R_*) , q, \la') =  - \eta(R_* , q, \la) $.

\item $\eta(-\rho(R_*) , q, \lb') =  - \eta(R_* , q, \lb) $. 

\item $\bar{d}(Y_3(-\rho(R_*), \la')) =  -\bar{d}(Y_3(R_*), \la) $.\ 

\item $\bar{d}(Y_3(-\rho(R_*), \lb')) =  - \bar{d}(Y_3(R_*), \lb) $. 

\end{enumerate}

\end{lemma}

Recall that $K=R_*$ with the choice $J_\alpha=D$ and $J_\beta=J$. With Lemma~\ref{lem:meta}, we see that the proof   of Theorem~\ref{thm:single2} is reduced to the following lemma, whose proof is postponed to the next subsection.

\begin{lemma}\label{lem:nonzero}  For all $r \not\equiv 0 \mod 7$, we have the following:

\begin{enumerate}

\item $\eta(K, 3, r  \la) \ne 0$. 

\item $\eta(K, 3, r \lb) =  0$.

\item $\bar{d}(Y_3(K), r \lb) \ne 0$.  

\end{enumerate} 

\end{lemma}

We finish the proof of Theorem~\ref{thm:single2} modulo the proof of Lemma~\ref{lem:nonzero}.  It is shown that for each metabolizer listed in Lemma~\ref{lem:cases},    the vanishing of the associated  slicing obstructions arising from  $\eta$--invariants  and $\overline{d}$--invariants, as provided by   Theorem~\ref{thm:obstruct},  leads to a contradiction of Lemma~\ref{lem:nonzero}.

\begin{enumerate}

\item  $\left< \la, \lb'\right>$.   If $\la$ is in the metabolizer, then the vanishing of the slicing obstructions includes the statement: $\eta(K , 3, \la) +    \eta(-\rho(K), 3, 0) = 0$.  Casson-Gordon invariants for trivial characters always vanish, so this contradicts  Lemma~\ref{lem:nonzero}~(1).

\item  $\left< \lb, \la'\right>$.     Here we use the element $\la'$ and the vanishing of the slicing obstruction to  conclude that $\eta(K , 3, 0) +    \eta(-\rho(K) , 3, \la') = 0$.   As in the last case, this contradicts Lemma~\ref{lem:nonzero}~(1) after using Lemma~\ref{lem:meta} to replace the $-\rho(K)$ term with one involving  $K$.

\item   $\left< \la, \la'\right>$.   This can be handled in the same way as  the previous two cases.

\item   $\left< \lb, \lb'\right>$.  Considering $\lb$, we would have  $\bar{d}(Y_3(K) , \lb) +    \bar{d}(Y_3(-\rho(K)) , 0) = 0$.   This falls to Lemma~\ref{lem:nonzero}~(3).

\item  $\left< \la + r \lb',  \lb + r^{-1} \la'\right>$, where $r \ne 0 \in \Z_7$.   In this case, this  leads to the equation  $\eta(K , 3, \la) +    \eta(K , 3, r \lb') = 0$.  This is addressed using Lemma~\ref{lem:nonzero}~(1) and (2).

\end{enumerate}


\subsection{Casson-Gordon and Heegaard Floer obstructions}\label{subsec:cg-hf}
In this subsection, we give a proof of Lemma~\ref{lem:nonzero}, which will complete the proof of Theorem~\ref{thm:single2}. The Casson-Gordon invariant we will use in this section is a {\it discriminant invariant}, which is determined by the value of $\eta$.  Details were presented in~\cite{MR1162937}.  The knots used there were almost identical to those we are considering, and~\cite{MR1162937} can serve as a complete reference.  (A similar calculation arises in~\cite[Appendix B]{MR3109864}.)
\smallskip

{\bf (1) $\boldsymbol {\eta(K, 3, r \la) \ne 0}$:}  The invariant $\eta$ is conjugation invariant. Therefore $\eta(K, 3, - \la) = \eta(K, 3,   \la) $.  Since $\la$ is a 2--eigenvector of the order three deck transformation, we have $\eta(K, 3,   \la) = \eta(K, 3,  2 \la) = \eta(K, 3,  4 \la)$.  Combining these, we have reduced the proof   to showing $ \eta(K, 3,   \la) \ne 0$. 

Observations of Gilmer~\cite{MR656619,MR711523} and Litherland~\cite{MR780587} relate the value of $\eta(K, 3, \la)  $ to that of $\eta(R(D,U), 3, \la) $ and classical invariants of $J$.  Since $R(D,U)$ bounds an evident  smooth slice disk $B$ and the element $\alpha$ itself bounds a smooth slice disk in the complement of $B$, we have that $\eta(R(D,U), 3, \la) = 0$.  Thus, we are reduced to considering the appropriate classical invariants of $J$.   Here is the result  we need.   The notation will be explained momentarily.

\begin{lemma}\label{lem:delta} If $\eta(K,3, \la) = 0 $,  then $\Delta_7(J)$ is a 7--norm. 
\end{lemma}

This is essentially~\cite[Corollary 6]{MR1162937}.  There the statement is presented as a slicing obstruction, but the obstruction is achieved by assuming that a specific Casson-Gordon invariant vanishes.  Also, a two-component link is being considered, but one of the components corresponds to the $\beta$ we are using here.

In Lemma~\ref{lem:delta}, $\Delta_7(J) $ is, by definition, $ \sqrt{\prod_{k=1}^{6}  \Delta_J(e^{2k\pi i/7})} = \sqrt{\big| H_1(Y_3(J))\big|}$, and the result assumes that    the square root is an integer.  In general, a positive integer $n$ is a $d$--norm if every prime factor of $n$ which is relatively prime to $d$ and has odd exponent in $n$, has odd order in $\Z_d^*$.   

In our case, $\Delta_J(t) = 5t^2 - 11t +5$ and a computation shows that $ \sqrt{\big| H_1(Y_3(J))\big|} = (13)(97)$.  The desired result is now immediate:  
$\gcd(7,13) =1$, $13$ has odd exponent 
in $(13)(97)$, and the order of 13 in $\Z_7^*$
is even ($13 \equiv -1 \mod 7$).  Thus, $\eta(K, 3, \la) \ne 0$ as desired.

\smallskip

{\bf (2) $\boldsymbol {\eta(K, 3, r \lb) = 0}$:} As in Case~(1), we first can reduce this to demonstrating that $\eta(K, 3,   \lb) = 0$.  Since $D$ is topologically slice, $K$ is also topologically slice, bounding a slice disk $B$, and $\beta$ bounds a slice disk in the complement of $B$.  It then follows from Casson-Gordon's original theorem that $\eta(K, 3, \lb) = 0$.  (We are using here the fact that the Casson-Gordon theorem applies in the topological locally flat setting, which is a consequence of Freedman's work~\cite{MR1201584,MR679066}.)

\smallskip

{\bf (3) $\boldsymbol{\bar{d}(Y_3(K), r \lb) \ne 0}$:} As with the previous cases, this can be reduced to the basic case that $\bar{d}(Y_3(K),   \lb) \ne 0$. The computation has three  parts, stated as a sequence of lemmas.  Our approach is  closely related to one in~\cite{MR3109864} and depends on a crucial calculation done there.  Note, however, that we must work with the $\bar{d}$--invariant, rather than with the $d$--invariant.  These results could be extracted from~\cite{MR3109864}  (see Theorems 6.2 and 6.5, along with Corollary 6.6 of~\cite{MR3109864}), but in our restricted setting, much more concise arguments are available.

The proof of the following statement includes an explanation as to why the two  homology groups  $H_1(Y_3(K)) $ and $H_1(Y_3(R(D,U))) $  can be identified.  We reduce the result to a computation related to $Y_3(R(D,U))$.

\begin{lemma}\label{lem:d-unknot}
$ {d}(Y_3(K), x) =d (Y_3(R(D,U)), x)  $ for all first homology classes $x$.
\end{lemma}

\begin{proof}  The knot $J$ can be converted into the unknot by changing negative crossings to positive.  Thus, there is a collection of unknots, $\{\gamma_i\}_{i=1,\ldots, r}$ (in fact, an unlink) in the complement of the natural genus one Seifert surface for $K$ such that $(-1)$--surgery on each has the effect of unknotting the band.  Each $\gamma_i$ bounds a surface in the complement of the Seifert surface. The curves $\gamma_i$ lift to $Y_3(K)$ to give a family of disjoint simple closed curves $\{\widetilde{\gamma}_{i,j}\}_{1\le i\le r, 1\le j\le 3}$.  By lifting the surfaces bounded by the $\gamma_i$ in the complement of the Seifert surface for $K$, we see that the curves $\widetilde{\gamma}_{i,j}$ are null-homologous and unlinked. 

It is now apparent  that $Y_3(R(D,U))$ can be built from $Y_3(K)$ by performing $(-1)$--surgery on all the curves in $\{\widetilde{\gamma}_{i,j}\}_{1\le i\le r, 1\le j\le 3}$.  There is a corresponding cobordism from $Y_3(K)$ to $Y_3(R(D,U))$  which is negative definite,  has diagonal intersection form, and the inclusions $Y_3(K)$  and  $Y_3(R(D,U))$ into the cobordism induce isomorphisms of the first homology.  Now, basic results of~\cite{MR1957829}   imply that $d(Y_3(K)) \ge   d(Y_3(R(D,U)))$.

We  also have that $J$ can be unknotted by changing positive crossings to negative.  The  argument just given yields the reverse inequality.
\end{proof}

\begin{lemma}\label{lem:trivial}  
$ {d}(Y_3(K), r \la) =  0$ for all $r \in \Z_7$.  In particular,  ${d}(Y_3(K), 0) = 0$.
\end{lemma}

\begin{proof} We consider $R(D, U)$ instead.  This knot is smoothly slice, so $Y_3(R(D,U))$ bounds a rational homology ball $W^4$.  The homology class $\la$ and its multiples are null-homologous in $W^4$, so the corresponding Spin$^c$--structure extends to $W^4$.  The vanishing of the $d$--invariant is then implied by results of~\cite{MR1957829}.
\end{proof}

We  now have our final lemma that completes the proof of Lemma~\ref{lem:nonzero}.

\begin{lemma}
$\bar{d}(Y_3(K), \lb) \ne 0$.  
\end{lemma}

\begin{proof} By Lemmas~\ref{lem:d-unknot} and \ref{lem:trivial} we can switch to considering the $d$--invariant rather than the $\bar{d}$--invariant, as follows.
\begin{equation*}
\begin{split}
\bar{d}(Y_3(K), \lb) &= {d}(Y_3(K), \lb) -  {d}(Y_3(K),0) \\
& =  {d}(Y_3(K), \lb) \\
& =  {d}(Y_3(R(D,U)), \lb).
\end{split}
\end{equation*}
The argument is then completed by quoting~\cite[Appendix A]{MR3109864}, where it is shown that ${d}(Y_3(R(D,U)), \lb) \le -3/2$.  (The statement in~\cite{MR3109864} refers to a homology class denoted $4 \widehat{x_2}$.  Notice that since $\lb$ is a  4--eigenvector,  ${d}(Y_3(R(D,U)), \lb) =   {d}(Y_3(R(D,U)), 4\lb)  = {d}(Y_3(R(D,U)), 2\lb)$.  Also, since the $d$--invariant is invariant under conjugation of Spin$^c$--structure, ${d}(Y_3(R(D,U)), x) = {d}(Y_3(R(D,U)), -x)$.  Thus, all $d$--invariants associated to nonzero elements in this eigenspace are equal.)
\end{proof}

\section{An infinite family of knots}\label{sec:infinite family}
Our goal in this section is to generalize the previous example in Section~\ref{sec:single knot} to build an infinitely generated free subgroup of $\calt/  (\text{Fix}(\rho)+\calt_\Delta)$, which will prove Theorem~\ref{thm:main2}.  

We now let the two bands in the Seifert surface in Figure~\ref{fig:knot} have $2n+1$ half-twists, and use the general notation $R_n$. For the choice of knots $J_\alpha$ and $J_\beta$, we will let $J_\alpha={\rm Wh}(T(2,-3),-1)$ be the positively clasped $(-1)$--twisted Whitehead double of the left-handed trefoil, and let $J_\beta={\rm Wh}(T(2,3),0)$ be the positively clasped untwisted Whitehead double of the right-handed trefoil. Notice that $J_\beta$ is topologically slice, hence so is $R_n(J_\alpha, J_\beta)$. Henceforth, we let $K_n=R_n(J_\alpha, J_\beta)$ for brevity.

The proof of Theorem~\ref{thm:main2} consists of selecting an appropriate set of positive integers $\caln$ for which we can prove that the set $\{R_n (J_\alpha, J_\beta)\}_{n \in \caln}$ represents a linearly independent set in  $\calt /( \rm{Fix}(\rho)+\calt_\Delta)$.  Computing the appropriate Heegaard Floer invariants of a branched cyclic cover  of    $ R_n (J_\alpha, J_\beta)$ relies on work of Cochran-Harvey-Horn \cite{MR3109864} and Cha \cite{arxiv:1910.14629}.

Recall that we let $Y_3(K)$ denote the 3--fold cover of $S^3$ branched over an arbitrary knot $K$.  The following is an elementary knot theoretic computation. 

\begin{lemma} $H_1(Y_3(K_n)) \cong \Z_{3n^2 +3n +1} \oplus \Z_{3n^2 +3n +1}$.
\end{lemma}

To simplify our computations, we would like to constrain the possible prime factorizations of $3n^2 +3n+1$.  This is provided by a number theoretic result, the proof of which is presented in the appendix.

\begin{theorem}\label{thm:primes} There is an infinite set of positive integers $\caln = \{n_i\}_{i\ge 1}$ such that for all $i$, $3n_i^2 + 3n_i +1 = p_{2i-1}p_{2i}$  where: {(1)} each $p_j$ is either an odd prime or equals 1; {(2)} if $j \ne l$ and $p_j \ne 1 $, then $p_j \ne p_l$; and {(3)} $1\in \caln$.
\end{theorem}

Our goal is to prove the following theorem, from which Theorem~\ref{thm:main2} immediately follows.

\begin{theorem} The set of knots  $\{K_n\}_{n \in \caln}$ is linearly independent in $\calt /(\text{Fix}(\rho)+\calt_\Delta)$.
\end{theorem}

Elementary group theory gives the following.

\begin{lemma}\label{lem:simple-argument} 
The set of knots $\{K_n\}_{n \in \caln}$ is linearly independent in $\calt /(\text{Fix}(\rho)+\calt_\Delta)$ if and only if the set of knots $\{K_n\cs -\rho(K_n)  \}_{n \in \caln}$ is linearly independent in $\calt/\calt_\Delta$.
\end{lemma}

This in turn is easily reduced to proving the following.

\begin{theorem}\label{thm:main3} Let $L$ be a knot with $\Delta_L(t) = 1$  and let
\[ K = 
\left(\cs_{n\in \caln} a_n(K_n \cs -\rho(K_n))\right) \cs L.
\] 
If $K=0\in \calt$ for some set of $a_n$ for which all but a  finite set of  $a_n$ are zero, then $a_n = 0$ for all $n$.
\end{theorem}

The rest of this section is devoted to proving Theorem~\ref{thm:main3}.

\subsection{Proof of Theorem~\ref{thm:main3}, First Step}\label{subsec:proof-first-step}

In this subsection we show how the argument is reduced to a statement about the $\overline{d}$--invariants and $\eta$--invariants  of $a_n(K_n \cs -\rho(K_n))$ for each $n\in \caln$.  

First, we give the following lemma.
\begin{lemma} \label{lem:vanish} 
$d( Y_3( K_n \cs -\rho(K_n)), 0) = 0$ and $\eta( K_n \cs -\rho(K_n) ,3 , 0) = 0$.
\end{lemma}
\begin{proof} The $d$--invariant and $\eta$--invariant are additive under connected sums.  

With regards to the $d$--invariant, the  spaces $Y_3(K_n) $ and $Y_3(\rho(K_n))$ are orientation-preserving diffeomorphic, and orientation reversal of a 3--manifold changes the sign of the $d$--invariant.

With regards to the $\eta$--invariant, from results going back to Gilmer~\cite{MR711523}  and Litherland~\cite{MR780587}, the value of $\eta(K_n \cs -\rho(K_n) ,3 , 0)$ is independent of $J_\alpha$ and $J_\beta$.   In the case that $J_\alpha$ and $J_\beta$ are both unknotted, $R_n(J_\alpha, J_\beta)$ is slice, and thus the Casson-Gordon invariant vanishes.
\end{proof}

The theorem below follows from Theorem~\ref{thm:obstruct} and Lemma~\ref{lem:vanish}.

\begin{theorem}\label{lem:metabolizer}  If $K = 0 \in \calt$, then there exists a subgroup $\calm \subset H_1(Y_3(K))$ for which: (1) $\big| \calm \big|^2 = \big| H_1(Y_3(K)) \big|$; (2) $\calm$ is a metabolizer for the linking form on $H_1(Y_3(K))$ and $\calm$ is invariant under the action of the order three deck transformation of $Y_3(K)$; (3) for all $z \in \calm$, $d(Y_3(K),z)) = 0$ and for all $z \in \calm$ of prime power order, $\eta(K,3, z) = 0$.
\end{theorem}

We write 
\[
S_n = a_n(K_n\cs -\rho(K_n)).
\]  
Observe that for each $n_i\in \caln$, $\Z_{3n_i^2+3n_i+1}\cong \Z_{p_{2i-1}} \oplus \Z_{p_{2i}}$, and hence there is a natural decomposition 
\[
H_1(Y_3(S_{n_i})) \cong  \big( ( \Z_{p_{2i-1}} \oplus \Z_{p_{2i}}) \oplus  (\Z_{p_{2i-1}} \oplus \Z_{p_{2i}}) \big)^{2|a_{n_i}|} .
\]
Also observe that since $\Delta_L(t)=1$, we have $H_1(Y_3(L))=0$. Since all $p_i$ are relatively prime, the metabolizer $\calm$ obtained from Theorem~\ref{lem:metabolizer} naturally splits into the direct sum of its $p$--primary components $\calm_p$:  
\[
\calm = \bigoplus_{i \ge 1} \left(\calm_{p_{2i -1}} \oplus \calm_{p_{2i}}\right),
\]
where $\calm_{p_{2i -1}} \oplus \calm_{p_{2i}}$ is a metabolizer for the linking form on $H_1(Y_3(S_{n_i}))$.
Since only a finite set of the $a_i$ are nonzero, only a finite set of the $ \calm_p$ are nonzero.  We now have the following corollary of Theorem~\ref{lem:metabolizer}.
\begin{corollary}\label{cor:vanishing}
If $K=0\in \calt$, then for all $z \in \calm_{p_{2i -1}} \oplus \calm_{p_{2i}}$,  $d\left(Y_3(S_{n_i}), z\right) = 0$, and for all $z \in \calm_{p_{2i -1}} \oplus \calm_{p_{2i}}$ of prime power order, $\eta \left( S_{n_i}, 3, z\right) = 0$.
\end{corollary} 


\subsection{Proof of Theorem~\ref{thm:main3}, Second Step}\label{subsec:proof-second-step}
Observe that for each $n_i=p_{2i-1}p_{2i} \in \caln$, at least one of $p_{2i-1}$ and $p_{2i}$ is greater than one.  By reordering, we can thus assume that for all $i$, $p_{2i -1} >1$.  In the appendix, we observe that  $n_1=1$, $p_1 =7$, $p_2=1$, and therefore $\calm_{p_1}\oplus \calm_{p_2}=\calm_7$.

In this subsection, first we will give a proof that if $K=0\in \calt$, then   $a_1=0$. Then, we will explain how that proof can be modified to show that $a_n=0$  for all $n\in \caln$.

\smallskip

\noindent{\bf Proof that $\boldsymbol{a_1=0}$:} Suppose $K=0\in \calt$. For brevity, let $a=a_1$ and $S=S_1$. Suppose $a\ne 0$. By changing the orientation if necessary, we may assume $a > 0$. 
Notice that 
\[
H_1(Y_3(S))=(\Z_7\oplus \Z_7\oplus\Z_7\oplus \Z_7)^a=\bigoplus_{i=1}^a\left(\langle x_i\rangle \oplus \langle y_i \rangle \oplus \langle x_i'\rangle \oplus \langle y_i' \rangle \right),
\]
where $x_i$ (respectively, $y_i$) is a lift of the curve $\alpha$ (respectively, $\beta$) to the $i$--th copy of $K_1=R_1(J_\alpha,J_\beta)$ in $Y_3(S)$, and  $x_i'$ (respectively, $y_i'$) is a lift of the curve $\beta$ (respectively, $\beta$) to the $i$--th copy of $\rho(K_1)$ in $Y_3(S)$.

On the homology group $H_1(Y_3(S))$ the deck transformation of order three acts. Viewing $H_1(Y_3(S))$ as a vector space over $\Z_7$, $H_1(Y_3(S))$ splits into the direct sum of the 2--eigenspace and the 4--eigenspace. We make a choice of orientation of $K_1$ such that the 2--eigenspace is generated by the $x_i$ and $y_i'$, and the 4--eigenspace is generated by $y_i$ and $x_i'$.

Since the metabolizer $\calm_{p_1}\oplus \calm_{p_2}$ is invariant under the action of the deck transformations of $Y_3(S)$, one can easily see that it splits into the direct sum of the 2--eigenspace  and the 4--eigenspace,  $E_2\oplus E_4$,
such that 
\[
E_2 \subset \bigoplus_{i=1}^a\left(\langle x_i\rangle \oplus \langle y_i' \rangle \right) \mbox{ and } 
E_4 \subset \bigoplus_{i=1}^a\left(\langle y_i\rangle \oplus \langle x_i' \rangle \right).
\]

\begin{lemma}\label{lem:eigenspace} If $K=0\in \calt$, then 
$E_2 = \bigoplus_{i=1}^a\langle x_i\rangle$ and $E_4=\bigoplus_{i=1}^a\langle x_i'\rangle.$
\end{lemma}

\begin{proof}
Recall, we are working now only with $n_1 = 7$ and will describe the extension to all $n_i$ later.  It suffices to show that $E_2 \subset \bigoplus_{i=1}^a\langle x_i\rangle$ and $E_4 \subset \bigoplus_{i=1}^a\langle x_i'\rangle$ since the order of the metabolizer $\calm_{p_1}\oplus \calm_{p_2}$, which is $7^{2a}$, is the same as that of the direct sum of $\bigoplus_{i=1}^a\langle x_i\rangle$ and $\bigoplus_{i=1}^a\langle x_i'\rangle$.

Suppose that $E_2$ is not contained in $\bigoplus_{i=1}^a\langle x_i\rangle$. Then, in $E_2$ there exists an element 
\[
h=(h_1, h_1', h_2, h_2', \ldots, h_a, h_a')\in \bigoplus_{i=1}^a\left(\langle x_i\rangle \oplus \langle y_i' \rangle \right) 
\]
such that $h_k'\ne 0$ in $\langle y_k'\rangle=\Z_7$ for some $1\le k\le a$. 

The Casson-Gordon invariant that we will use in this section is the Casson-Gordon {\it signature invariant}, which we also denote by $\eta$. 
Let $\sigma_r(J_\alpha)$ denote the Levine-Tristram signature function of $J_\alpha$ evaluated at $e^{2\pi r\sqrt{-1}}$.  As described earlier, results of Gilmer~\cite{MR711523} and Litherland~\cite{MR780587} describe how the value of $\eta(S,3,h)$ is determined by the values of $\eta(R_1(U, J_\beta),h)$ along with values of $\sigma_r(J_\alpha)$ for specified values of $r$.      Because   $J_\beta$ is topologically slice, Casson-Gordon invariants cannot distinguish  $R_1(U,J_\beta)$ from $R_1(U,U)$, and for this knot all possible Casson-Gordon invariants vanish.    One concludes that  the relevant values of $\eta(R_1(U,J_\beta), h)$ will vanish.  Combining these observations yields 
\[
\eta(S,3,h)= \sum_{i=1}^a \epsilon_i\left( \sigma_{b_i/7}(J_\alpha) +\sigma_{2b_i/7}(J_\alpha)+\sigma_{4b_i/7}(J_\alpha) \right),
\]
where $b_i\in \Z_7$, and $\epsilon_i=0$ if $h_i'=0$ and $\epsilon_i =1$ if $h_i'\ne 0$. The knot $J_\alpha$ has the same Seifert form as the right-handed trefoil, and therefore
\[
\sigma_r(J_\alpha) = 
\begin{cases}	 
0 	&0\le r< \frac13\\ 
-2	& \frac13<r \le \frac12. 
\end{cases} 
\]
Therefore, we have $\eta(S,3,h)\le 0$. We are assuming that $h'_k \ne 0$, so $\epsilon_k=1$ and $b_k \ne 0 \in \Z_7$. Regardless of the value of $b_k$, 
\[
\sigma_{b_k/7}(J_\alpha) +\sigma_{2b_k/7}(J_\alpha)+\sigma_{4b_k/7}(J_\alpha) < 0.
\] 
It follows that $\eta(S,3,h)<0$, which contradicts Corollary~\ref{cor:vanishing}. One can also show $E_4 \subset \bigoplus_{i=1}^a\langle x_i'\rangle$, similarly.
\end{proof}

\begin{lemma}  If $K=0\in \calt$, then $a_1 = 0$.
\end{lemma}

\begin{proof}
By Lemma,~\ref{lem:eigenspace}  we obtain
\[
\calm_{p_1}\oplus \calm_{p_2}=\bigoplus_{i=1}^a\left(\langle x_i\rangle \oplus \langle x_i'\rangle \right).
\]
Therefore, the homology class $x_1$ is in $\calm_{p_1}\oplus \calm_{p_2}$, and hence $4x_1\in \calm_{p_1}\oplus \calm_{p_2}$. By Corollary~\ref{cor:vanishing}, we have $d(Y_3(S), 4x_1)=0$. By Sato \cite[Theorem~1.2]{Sato:2019a}, a genus one knot with vanishing Ozsv\'{a}th-Szab\'{o} $\tau$--invariant is $\nu^+$--equivalent to the unknot. The knot $J_\alpha$ has genus one and $\tau(J_\alpha)=0$ by \cite[Theorem~1.5]{MR2372849}; it follows that $J_\alpha$ is $\nu^+$--equivalent to the unknot. Now by Theorems~1.3 and 2.7 of \cite{KKK:2019aa}, $d(Y_3(S), 4x_1)=d(Y_3(S'),4x_1)$ where $S'$ is the knot obtained from $S$ by replacing $J_\alpha$ by the unknot. Therefore, $d(Y_3(S'),4x_1)=0$. But in \cite[p.~2141]{MR3109864} Cochran-Harvey-Horn showed that $d(Y_3(S'), 4x_1)<0$. This leads us to a contradiction, and completes the proof for $a_1=0$. 
\end{proof}

{\bf General proof that    $\boldsymbol{a_{n_j} = 0}$ for $\boldsymbol{n_j\in \caln}$:} The proof for $a_{n_j}=0$ for other $n_j\in \caln$ is easily obtained by making the following key modifications of the above proof for $a_1=0$:

\begin{enumerate}

\item For brevity, let $n=n_j$, $a=a_n$, and $S=S_n$. Replace $p_1$ and $p_2$ by $p_{2j-1}$ and $p_{2j}$, respectively. Replace $K_1=R_1(J_\alpha,J_\beta)$ by $K_n=R_n(J_\alpha,J_\beta)$.

\item $H_1(Y_3(S))=\Z_{3n^2+3n+1}^{4a}=(\Z_{p_{2j-1}}\oplus \Z_{p_{2j}})^{4a}$. Notice that   each of $\langle x_i\rangle$, $\langle y_i \rangle$, $\langle x_i'\rangle$, and $\langle y_i' \rangle$ for $1\le i\le a$ is isomorphic to $\Z_{3n^2+3n+1}$. 

\item For $x\in \Z_{3n^2+3n+1}$, let $x^*$ denote the multiplicative inverse of $x$ in $\Z_{3n^2+3n+1}$, if it exists. Replace the 2--eigenspace  and the  4--eigenspace by $n^*(n+1)$-- and $(n+1)^*n$--eigenspaces, respectively. Then we obtain $\calm_{p_{2j-1}}\oplus \calm_{p_{2j}}=E_{n^*(n+1)}\oplus E_{(n+1)^*n}$. 

\item In the  proof Lemma~\ref{lem:eigenspace} for $n=1$, the order of $h$ was 7, a prime. But now the order of $h$ in $E_{n^*(n+1)}$ is a factor of $p_{2j-1}p_{2j}$, possibly not a prime. To use the vanishing criterion for the Casson-Gordon invariant, if necessary, replace $h$ by a multiple of $h$ such that $h_k'\ne 0$ and $h$ is of prime order $p$ where $p$ is either $p_{2i-1}$ or $p_{2i}$. 

\item In the  proof Lemma~\ref{lem:eigenspace}, replace $\sigma_{b_i/7}(J_\alpha) +\sigma_{2b_i/7}(J_\alpha)+\sigma_{4b_i/7}(J_\alpha)$ by $\sigma_{b_i/p}(J_\alpha) +\sigma_{cb_i/p}(J_\alpha)+\sigma_{c^2b_i/p}(J_\alpha)$ where $c=n^*(n+1)\ne 0 \in \Z_{3n^2+3n+1}$.  For the prime $p$, there exists $b\in \Z_p$ so that the set $\{ b/p, cb/p, c^2b/p\}$ contains a value at which the Levine-Tristram signature of $J_\alpha$ is strictly negative. If necessary, replace $h$ by a multiple of $h$ such that $b_k=b\in \Z_p$.

\item Replace $4x_1\in \calm_{p_1}\oplus \calm_{p_2} $ by $2^*x_1\in \calm_{p_{2j-1}}\oplus \calm_{p_{2j}}$. In Theorem~4.2 of \cite{arxiv:1910.14629}, Cha showed that $d(Y_3(S'), 2^*x_1)<0$. 

\end{enumerate}

\section{Conjectures}\label{sec:conjectures}

The map $\rho$ induces   homomorphisms on  many  subgroups and quotients of subgroups related to $\calc$.  In each case, we will continue to denote the map by $\rho$.   

In~\cite{MR3109864}, Cochran, Harvey and Horn defined a {\it bipolar filtration} of the knot concordance group, which, when restricted to $\calt$, gives  a filtration 
\[
0 \subset \cdots \subset \calt_{n+1} \subset \calt_n \subset \cdots \calt_0 \subset\calt.
\]
Let $\calt_{n, \Delta}= \calt_n/(\calt_n \cap \calt_\Delta)$; notice that $\rho$ induces an involution on this quotient.

The first conjecture  seems likely, based on~\cite{Cha:2017aa}.

\begin{conjecture} For all $n \ge 1$, the quotient $\calt_{n, \Delta}/\text{Fix}(\rho)$ contains an infinitely generated free subgroup.
\end{conjecture}  

The next conjecture also seems likely, but it it not clear that any currently available tools can address it.   

\begin{conjecture}
The quotient $\calt_\Delta/\text{Fix}(\rho)$ contains an infinitely generated free subgroup.
\end{conjecture}

Finally, each of these conjectures can be modified to consider two-torsion.  It was proved in~\cite{MR3466802} that $\calt$ contains an infinite set of elements of order two, as does   $\calt / \calt_\Delta$.  These knots were all reversible.   

\begin{conjecture} There exists a knot $ K \in \calt$ such that $2K = 0$  but  $K \ne  \rho(K)$ in $\calt$.
\end{conjecture}

\appendix

\section{Primes}\label{sec:appendix}

We wish to prove the following, stated as Theorem~\ref{thm:primes} above.  

\begin{theorem-app}\label{thm:app1}     There is an infinite set of positive integers $\{n_i\}_{i\ge 1}$ such that for all $i$, $3n_i^2 + 3n_i +1 = p_{2i-1}p_{2i}$  where: {(1)} each $p_j$ is either an odd prime or equals 1, and {(2)} if $j \ne l$ and $p_j \ne 1 $, then $p_j \ne p_l$.
\end{theorem-app}

The proof is based on the following theorem of Lemke Oliver~\cite{MR2860953}.  (The meaning of $\Gamma_G$  in the statement of the theorem will be mentioned in the following proof.)

\begin{theorem-app} \label{thm:oliver} If  $G(x)  = c_2x^2 + c_1x +c_o \in\Z[x]$ is irreducible, with $c_2 > 0$ and $\Gamma_G \ne 0$, then there exist infinitely many positive integers $n$ such that $G(n)$ is square free and has at most two distinct prime factors. 

\end{theorem-app}

\begin{proof}[Proof of Theorem~\ref{thm:app1}]  Let $f(n) = 3n^2 + 3n+1$ and note that $f(n)$ is odd for all $n \in \Z$.  Let $n_1 = 1$; then we have $p_1 = 7$ and $p_2 = 1$.  Assume that a set of integers $\{n_j\}_{j=1}^k$  that satisfies the condition of the theorem has been selected.   We now show how $n_{k+1}$ can be chosen.

Let $P = \prod_{i=1}^{2k} p_i$.   Define $g(m) = f(Pm - 1 )$.  This can be rewritten as 
\[
g(m) = 3P^2 m^2 -3Pm +1.
\]
Since $g(m)$ is obtained from the irreducible polynomial $f(n)$ by a linear change of coordinates, $g(m)$ is irreducible and   Theorem~\ref{thm:oliver} can be applied to find an $m_0$ for which $g(m_0)$ factors as $p_{2k+1}p_{2k+2}$.  We let $n_{k+1} = Pm_0 -1$.   Notice that no prime factor of $P$ is a divisor of $g(m)$ for any $m$, and thus  $p_{2k+1}$ and $ p_{2k+2}$ are distinct from all the primes $p_i$ for $i \le 2k$.

Finally, we need to mention the quantity $\Gamma_G$.  Without going into details, $\Gamma_G = 0$ precisely when $G(n)=0$ has two solutions modulo  2.  But in our case, modulo 2, $g(m) = m^2 + m +1$, which is irreducible.
\end{proof}


\bibliography{BibTexComplete.bib}
\bibliographystyle{plain}	

\end{document}